\documentclass{article}

\usepackage{graphicx}
\usepackage{latexsym}
\usepackage[fleqn]{amsmath}
\usepackage{amssymb}
\usepackage{amsbsy}
\usepackage{nishimura-pp2022}
\usepackage{caption}
\usepackage{subcaption}

\begin{document}

\title{Control Barrier Functions for Stochastic Systems and Safety-critical Control Designs}
\author{Y\^uki Nishimura\footnotemark[1] \ and Kenta Hoshino\footnotemark[2]}
\date{September 19, 2022; Revised April 17, 2024\footnotemark[3]}

\renewcommand{\thefootnote}{\fnsymbol{footnote}}
\footnotetext[1]{Kagoshima University}
\footnotetext[2]{Kyoto University}
\footnotetext[3]{This work has been submitted to the IEEE for possible publication. Copyright may be transferred without notice, after which this version may no longer be accessible.}

\renewcommand{\thefootnote}{\arabic{footnote}}
\maketitle

\begin{abstract}
In recent years, the analysis of a control barrier function has received considerable attention because it is helpful for the safety-critical control required in many control application problems. While the  extension of the analysis to a stochastic system studied by many researchers, it remains a challenging issue. In this paper, we consider sufficient conditions for reciprocal and zeroing control barrier functions ensuring safety with probability one and design a control law using the functions. Then, we propose another version of a stochastic zeroing control barrier function to evaluate a probability of a sample path staying in a safe set and confirm the convergence of a specific expectation related to the attractiveness of a safe set. We also show a way of deisgning a safety-critical control law based on our stochastic zeroing control barrier function. Finally, we confirm the validity of the proposed control design and the analysis using the control barrier functions via simple examples with their numerical simulation. 
\end{abstract}

\section{Introduction}

In recent control application problems, opportunities for demands for sophisticated machine behavior and machine-to-human contact have increased. Because the problems require machine behavior to stay within a safe range for the machine itself and humans, the center of the control design guidelines is changing from stability to {\it safety}. The transition is theoretically realized by the change from stabilization based on a control Lyapunov function to safety-critical control based on a control barrier function (CBF) {\cite{ames2017,ames2019}}. In the last few years, research results on a CBF have been actively reported { with various control application problems, and in reality, the simple realization of seemingly complicated commands \cite{ames2019,shimizu2022}, and human assist control \cite{furusawa2021,nakamura2019,tezuka2022} are being promoted}.

In the context of a CBF, the control objective is to make a specific subset, which is said to be a {\it safe set}, on the state space invariance forward in time (namely, forward invariance \cite{ames2019}). { There are various types of CBFs, the most commonly used currently are a reciprocal control barrier function (RCBF) \cite{ames2019,furusawa2021,nakamura2019} and a zeroing control barrier function (ZCBF) \cite{ames2019,shimizu2022,tezuka2022}: the RCBF is a positive function that diverges from the inside of the safe set toward the boundary, while the ZCBF is a function that is zero at the boundary of the safe set. The RCBF has a form that is easy to imagine as a barrier, while the ZCBF is defined outside the safe set, allowing the design of control laws with robustness.}

{The} aforementioned literature \cite{ames2017,ames2019,shimizu2022,furusawa2021,nakamura2019,tezuka2022} focuses on systems without stochastic disturbances. Because a stochastic disturbance often affects a real system, a safe set is desirable to maintain invariance even when {influenced by the} disturbance. Recently, various types of CBF-based stochastic safety-critical control have been proposed in \cite{prajna2007,santoyo2019,wisniewski2021,jagtap2021,salamati2023,clark2021,wang2021,bai2022,nejati2022}. 
Jagtap et al. \cite{jagtap2021} conducts a systematic and detailed study, and then it is developed into a data-driven framework by Salamati and Zamani \cite{salamati2023}. Prajna et al. \cite{prajna2007} provides a safety verification procedure, and then it is developed to control design procedure by Santoyo et al. \cite{santoyo2019}. Wisniewski and Bujorianu \cite{wisniewski2021} also discuss in detail safety in an infinite time-horizon named $p$-stability. {Bai et al. \cite{bai2022} analyzes a probability for a trajectory to reach a target set, which is a subset of a safe set. Nejati et al. \cite{nejati2022} develop a compositional approach for constructing CBFs for stochastic hybrid systems, which forms an excellent theory in terms of applications because they use numerical methods such as the sum-of-squares optimization program under the free design of safe sets. }

{On the other hand, the CBF approach is closely related to a control Lyapunov function (CLF), which immediately provides a stabilizing control law from the CLF, as in Sontag \cite{sontag} for deterministic systems and Florchinger \cite{florchinger} for stochastic systems. Therefore, in the CBF approach, the derivation of a safety-critical control law immediately from the CBF is also important. For this discussion, the problem setting in which the safe set is coupled with the CBF is appropriate, as in Ames et al. \cite{ames2019}.} {The stochastic version of the Ames's et al.'s result is recently discussed} by Clark \cite{clark2021}; he insists that his RCBF and ZCBF guarantee the safety of a set with probability one. At the same time, Wang et al. \cite{wang2021} analyze the probability of a time when the sample path leaves a safe set under conditions similar to Clark's ZCBF. Wang et al. also claim that a state-feedback law achieving safety with probability one often diverges toward the boundary of the safe set; the inference is also obtained from the fact that the conditions for the existence of an invariance set in a stochastic system are strict and influenced by the properties of the diffusion coefficients \cite{nishimura2016scl}. {This argument is in the line of stochastic viability by Aubin and Prato \cite{aubin1995}. For CBFs, Tamba et al.~\cite{tamba2021} provides sufficient conditions for safety with probability one, which require difficult conditions for the diffusion coefficients.} Therefore, we need to reconsider a sufficient condition of safety with probability one, and we also need to rethink the problem setup to compute the safety probability obtained by a bounded control law.

In this paper, we propose a way of analyzing safety probability for a stochastic system via a CBF approach. The contributions of this paper are as follows. First, we propose an almost sure reciprocal control barrier function (AS-RCBF) ensuring the safety of a set with probability one, which is considered as a stochastic version of an extended RCBF in \cite{nakamura2019}; see also \cite{furusawa2021} (and note that the condition is relaxed {around the boundary of the safe set} compared with an RCBF in \cite{ames2017}). Second, we propose an almost sure zeroing control barrier function (AS-ZCBF) satisfying an inequality somewhat different from the one in \cite{clark2021}. Then, we suggest a new stochastic ZCBF for {calculating a probability that a trajectory achieves a designed subset of a safe set before leaving the safe set}. Our stochastic ZCBF satisfies an inequality, which differs from the previous results in \cite{prajna2007,santoyo2019,wisniewski2021,clark2021,wang2021,bai2022,nejati2022} because the inequality directly includes the diffusion coefficients. { In the procedure, we also provide control design strategies using AS-RCBF/AS-ZCBF and our stochastic ZCBF.} In addition, we demonstrate our stochastic ZCBF is available for stochastic systems including input constraints by simple examples.

{The rest of this paper is organized as follows. In} Section~\ref{sec:preliminary}, we define mathematical notations, a target system, and a global solution used in this paper. 
In Section~\ref{sec:motivation}, by considering a simple example, we confirm that a stochastic system is generally difficult to have a safe set invariance with probability one. 
In Section~\ref{sec:main}, first, we propose an {AS-RCBF and an AS-ZCBF} ensuring the invariance of a safe set with probability one. Second, we design a safety-critical control ensuring the existence of an AS-RCBF and an AS-ZCBF and show that the controller diverges towards the boundary of a safe set. Third, we construct a new type of a {\it stochastic ZCBF} clarifying a probability for the invariance of a safe set and showing the convergence of a specific expectation related to the attractiveness of a safe set from the outside of the set. 
In Section~\ref{sec:example}, we confirm the usefulness of the proposed functions and the control design via simple examples with numerical simulation. 
Section~\ref{sec:conclusion} concludes this paper. 

\section{Preliminary}\label{sec:preliminary}

\subsection{Notations}

Let $\R^n$ be an $n$-dimensional Euclidean space and especially $\R := \R^1$. A Lie derivative of a smooth mapping ${y}: \R^n \to \R$ in a mapping $F = (F_1,\ldots,F_q): \R^n \to \R^{n \times q}$ with $F_1,\ldots,F_q: \R^n \to \R^n$ is denoted by 
\begin{align}
\lie{F}{{y}}(x) = \left( \pfrac{{y}}{x} F_1(x), \ldots, \pfrac{{y}}{x} F_q(x) \right).
\end{align}
For constants $a,b>0$, a continuous mapping $\alpha:[-b,a] \to \R$ is said to be an extended class $\mathcal{K}$ function if it is strictly increasing and satisfies $\alpha(0)=0$. { A class $\mathcal{K}$ function $\alpha$ is said to be of $\mathcal{K}_\infty$ if $\lim_{s\to\infty}\alpha(s)=\infty$. If a function $\alpha:\R^n \to \R$ is continuously differentiable for $r$-times, we state it as ``$\alpha$ is $C^r$.''} The boundary of a set $\mathcal{A}$ is denoted by $\partial \mathcal{A}$.

Let $(\Omega,\mathcal{F},\{\mathcal{F}_t\}_{t \ge 0},\mathbb{P})$ be a filtered probability space, where $\Omega$ is the sample space, $\mathcal{F}$ is the $\sigma$-algebra {of $\Omega$}, $\{\mathcal{F}_t\}_{t \ge 0}$ is a filtration of $\mathcal{F}$ and $\mathbb{P}$ is a {probability} measure. In the filtered probability space, { $\pr{A|A_o}$ is the conditional} probability of event $A$ {conditioned on event $A_o$, $\ex{y|A_o}$ is the conditional} expectation of {the} random variable $y$ {conditioned on event $A_o$}, and {$W_t$} is a $d$-dimensional standard Wiener process. For a process ${X_t} \in \R^n$ with an initial state ${X_t}=x_0$, we often use the following notations $\pri{x_0}{A} = \pr{A|{X_0}=x_0}$ and $\exi{x_0}{y}=\ex{y|{X_0}=x_0}$. The minimum of $a,b \in \R$ is described by $a \wedge b:= \min(a,b)$. The differential form of an It\^o integral of $f:\R^n \to \R^n$ over {$W_t$} is represented by $f(x) d{W_t}$. The trace of a square matrix {{$Q$}} is denoted by $\mathrm{tr}[{{Q}}]$. 

\subsection{Target system, the related functions, and a global solution}

In this subsection, we describe a target system, the related functions frequently used throughout the paper, and the definition of a solution in global time.

The main target of this paper is the following stochastic system
\begin{align}\label{eq:sys-sto-gen}
d{X_t} &= \{ f({X_t}) + g({X_t}) (u_o({X_t}) + u{(t)}) \} dt + \sigma({X_t}) d{W_t},
\end{align}
where { $X_t \in \R^n$ is a state vector, $u_{o}:\R^n \to \R^m$ is a pre-input assumed to be a continuous state-feedback, $u \in U \subset \R^m$ is a compensator for safety-critical control, where $U$ denotes an acceptable control set, and maps $f: {\R^{n \times n}}$ and $g: \R^n \to {\R^{n \times m}}$ and} $\sigma: \R^n \to {\R^{n \times d}}$ { are all} assumed to be locally Lipschitz. The local Lipschitz condition of $f$, $g$ and $\sigma$ implies that there exists a stopping time $T>0$ such that {$(X_t)_{t<T}$} is {the} maximal solution to the system.

For simplicity, we further define some functions. For a $C^2$ mapping $y:M \to \R${, where $x \in $} $M \subset \R^n$, let{ting}
\begin{align}
&L^D_{f,g}(u,u_o(x),y(x)) := (\lie{f}{y})(x) + (\lie{g}{y})(x) (u+u_o(x)),\\
&L^I_{\sigma}(y(x)) := \frac12 \mathrm{tr} \left[ \sigma(x) \sigma(x)^T \left[ \pfrac{}{x} \left[ \pfrac{y}{x}\right]^T \right](x) \right],
\end{align}
{we consider an infinitesimal operator $\mathcal{L}$ in \cite{khasminskii2012} satisfying}
\begin{align}
&\mathcal{L}_{f,g,\sigma}(u,u_o(x),y(x)) := L^D_{f,g}(u,u_o(x),y(x)) + L^I_{\sigma}(y(x))
\end{align}
and
\begin{align}\label{eq:H}
H_{\sigma}(h(x)) := \frac12 \lie{\sigma}{h}(x) (\lie{\sigma}{h}(x))^T. 
\end{align}
For a mapping $v:M \to (0,\infty)$ smooth in $M \subset \R^n$, we often consider the relationship 
\begin{align}\label{eq:rel-bh}
-(v(x))^{-2} L^D_{f,g}(u,u_o(x),v(x)) = L^D_{f,g}(u,u_o(x),(v(x))^{-1}).
\end{align}

Moreover, based on \cite{nishimura2018automatica}, we describe the following notion meaning the existence of a global solution in forward time for the system \eqref{eq:sys-sto-gen}:
\begin{definition}[FIiP and FCiP{; a slight modification of (C2) in \cite{nishimura2018automatica}}]\label{def:fcip}
Let an open subset $M \subset \R^n$ and the system \eqref{eq:sys-sto-gen} be considered with $u=\phi(x)$, where $\phi: M \to \R^n$ is a continuous mapping. If a $C^2$ mapping $Y: M \to [0,\infty)$ { is proper; that is, for any $L \in [0,\infty)$, any sublevel set $\{x \in M | Y(x) \le L\}$ is compact}, and a continuous mapping $\psi: [0,\infty) \times (0,1) \to [0,\infty)$ both exist for every $x_0 \in M$ such that
\begin{align}\label{eq:fcip}
\pri{x_0}{\forall t \in [0,l],\ Y({X_t}) \le \psi(l,\epsilon)} \ge 1- \epsilon
\end{align}
holds for all $l \in [0,\infty)$ and all $\epsilon \in (0,1{]}$, then the system is said to be forward invariance in probability (FIiP) in $M$. In addition, if $M =\R^n$ and $Y(\cdot)=|\cdot|$, the system is said to be forward complete in probability (FCiP). \eod
\end{definition}

\begin{theorem}{\it (A slight modification of Proposition~17 in \cite{nishimura2018automatica}):}\label{thm:fcip}
Let us consider the system \eqref{eq:sys-sto-gen}, an open subset $M \subset \R^n$, a continuous mapping $\phi: M \to \R^n$ and an initial condition $x_0 \in M$. If there exists a proper and $C^2$ mapping $Y: M \to [0,\infty)$ such that
\begin{align}
\mathcal{L}_{f,g,\sigma}(\phi(x),u_o(x),Y(x)) \le c_1 Y(x) + c_2
\end{align}
is satisfied for all $x \in M$ and for some $c_1 \in [0,\infty)$ and $c_2 \in [0,\infty)$, then the system with $u=\phi(x)$ is FIiP in $M$. In addition, if $M=\R^n$, the system is FCiP. \eot
\end{theorem}

Definition~\ref{def:fcip} and Theorem~\ref{thm:fcip} are the same as (C2) and Proposition~17 in \cite{nishimura2018automatica}, respectively, except for two differences; $x$ is restricted in $M$ and $Y$ is allowed to be not positive definite in this paper, while $x$ is allowed to be in $\R^n$ and $Y$ is restricted to be positive definite in the literature. Because $M$ is an open set and $Y$ is non-negative and proper, $Y(x) \to \infty$ always holds as $x \to \partial M$ \cite{nakamura2019}. The positive definiteness of $Y$ is required for stability analysis and it can be omitted for just analyzing forward invariance and completeness. Therefore, Theorem~\ref{thm:fcip} is straightforwardly proven via the proof of Proposition~17 in \cite{nishimura2018automatica} by replacing $\R^n$ by $M$.

Note that, FIiP in $M$ implies that the probability of $Y({X_t}) \to \infty$ is infinitesimal; this estimate that the solution {$X_t$} starting at $x_0 \in M$ stays in $M$ with probability $1-\epsilon$ for arbitrarily small $\epsilon$.

\section{Motivating Example}\label{sec:motivation}

\subsection{An example of safety-critical control for a deterministic system}

Firstly, let us consider a safety-critical control problem { based on Theorem~2 in \cite{ames2019}; that is,} {a}ssume that a set {$\tilde{\chi}$} $\subset \R^n$ is a superlevel set of a continuously differentiable mapping $h:\tilde{\chi} \to \R$ which satisfies $h(x) \ge 0$ for all $x \in \tilde{\chi}$, $h(x)=0$ for any $x \in \partial \tilde{\chi}$ and $\partial h/ \partial x \neq 0$ for all $x \in \tilde{\chi}$. For a system $\dot{x}=f(x)+g(x)u$, if there exists a compensator $u=\phi(x)$ such that (there exists a global solution in forward time in $\R^n$ and) there exists an extended class $\mathcal{K}_\infty$ function $\bar{\alpha}: \tilde{\chi} \to \R$ satisfying
\begin{align}\label{eq:con-cbf-ames}
L^D_{f,g}(\phi(x),{0},h) \ge -\bar{\alpha}(h(x)),
\end{align}
then any solution $x(t)$ starting at $x_0 \in$ {$\tilde{\chi}$} satisfies $x(t) \in \tilde{\chi}$ for all $t \in [0,\infty)$. 

As an example, we consider
\begin{align}\label{eq:sys-det}
\dot{x} = u,
\end{align}
where $x \in \R$, $u \in U =\R$, and $x(0)=x_0> \alpha {\ \ge\ } 0$. Now we let a safe set as $\tilde{\chi} = [\alpha,\infty)$; that is, we aim to design $u$ so that $x(t) \in \tilde{\chi}$ is satisfied for all $x_0 \in $ $\tilde{\chi}$ and all $t \ge 0$. If, for $h_s(x)=x-\alpha$ and $\gamma>0$, we design $u=\phi_{h_s}{:=-\gamma (x-\alpha),}$ then we obtain {$L^D_{0,1}(\phi_{h_s}(x),0,h_s(x)) \ge -\gamma h_s(x),$} which satisfies \eqref{eq:con-cbf-ames} by considering $h=h_s$ and $\bar{\alpha}(h_s)=\gamma h_s$. Therefore, {$\tilde{\chi}$} becomes safe by $u=\phi_{h_s}$.

The same result as above is also derived by using {$B_s(x) := (h_s(x))^{-1}=(x-\alpha)^{-1},$} provided that the function satisfies
\begin{align}\label{eq:ex-mot-rcbf}
L^D_{0,1}(\phi_{h_s}(x), {0},B_s(x)) \le \gamma B_s(x)
\end{align}
in $x \in \tilde{\chi}$; the condition is somewhat different from an RCBF in \cite{ames2019} and similar to an extended RCBF in \cite{furusawa2021,nakamura2019}. (Strictly, an extended RCBF further requires $B_s$ to be proper, and it is defined for a time-varying system{.})  Note from this discussion that the extended RCBF is inferred to be a counterpart concept to the ZCBF. Then, \eqref{eq:ex-mot-rcbf} implies that the value of the extended RCBF is allowed to be large, but is guaranteed not to be out of the safe set in finite time.

{A} ZCBF $h_s(x)$ is defined in whole $\R$ and $B_s(x)$ is bounded just inside of {$\tilde{\chi}$} $\setminus \partial $ {$\tilde{\chi}$} $\subset \R^n$. Therefore, $h_s(x)$ is generally useful in the viewpoint of robust control because modeling and measurement errors often cause an initial value outside of {$\tilde{\chi}$}.

\subsection{Trying extension of safety-critical control to a stochastic system}

In this subsection, we try to extend the discussion in the previous subsection to a stochastic system. 

A stochastic version of a CBF is discussed in \cite{clark2021} and \cite{wang2021}. Roughly speaking, Theorem~3 in \cite{clark2021} claims that, considering a stochastic system \eqref{eq:sys-sto-gen} with $x_0 \in {\tilde{\chi}}$, if the condition \eqref{eq:con-cbf-ames} is replaced by
\begin{align}\label{eq:con-zcbf-clark}
\mathcal{L}_{f,g,\sigma}(\phi(x),{0},h(x)) \ge -h(x)
\end{align}
for $x \in \tilde{\chi} \setminus \partial \tilde{\chi}$, then any solution satisfies $X_t \in \tilde{\chi}$ for all $t \ge 0$ {\it with probability one}. However, another previous result in \cite{wang2021} shows that  the first exit time of $X_t$ from $\tilde{\chi}$ has a finite value with a non-zero probability. The claim implicitly implies that the probability of exiting $\tilde{\chi}$ is generally not zero even if \eqref{eq:con-zcbf-clark} holds. 

Here, we consider the answer to the above contradiction by considering a safety-critical control for a stochastic system
\begin{align}\label{eq:sys-sto}
d{X_t} =  u dt + c d{W_t},
\end{align}
which is the same form as \eqref{eq:sys-det} except for the existence of the diffusion term $c d{W_t}$, where $c \neq 0$. 

As with the previous subsection, we consider $h_s(x)=x-\alpha$ as a candidate for a ZCBF. Because the Hessian of the function is always zero, we obtain {$L^I_{c}(h_s)= 0.$} This implies that, setting $u=\phi_{h_s}$ results in
\begin{align}
\mathcal{L}_{0,1,c} (\phi_{h_s}(x), {0},h_s(x)) \ge - \gamma h_s(x),
\end{align}
which satisfies the condition \eqref{eq:con-zcbf-clark} with $\gamma=1$. However, a solution to the resulting system {$d {X_t} = -\gamma	({X_t}-\alpha) dt + c d{W_t} $}
has a non-zero probability to escape $\tilde{\chi}$ even if $x_0 \in \tilde{\chi}$ {because the solution is 
\begin{align}
X_t=\alpha +(x_0-\alpha) e^{-\gamma t} + c \int_0^t e^{-\gamma(t-s)} dW_t
\end{align}
(see Sec.~3.5 in \cite{mao2007}). For} example, if $x_0 = \alpha$, { $X_t - \alpha$ follows the normal distribution with the mean zero and the variance $c^2(1-e^{-2\gamma t})/(2\gamma)$; that is, $X_t -\alpha$ is possible to be negative}. This example implies that the condition \eqref{eq:con-zcbf-clark} ensures {$\tilde{\chi}$} to be safe ``in probability''.

On the other hand, considering $B_s(x)=(x-\alpha)^{-1}$, the Hessian does not vanish and results in {$L^I_{c} = c^2 (x-\alpha)^{-3};$}
hence, we can estimate that $B_s$ yields an answer to the control problem different from $h_s${; that is,} a compensator $u=\phi_{B_s}(x)$ {$:= -\gamma (x-\alpha) + c^2 B_s(x)$} yields
\begin{align}
\mathcal{L}_{0,1,c} (\phi_{B_s}(x),u_o(x),B_s(x)) \le \gamma B_s(x).
\end{align}
Because the compensator diverges at $\partial $ {$\tilde{\chi}$}, it may have the potential to cage the solution $x$ in {$\tilde{\chi}$} with probability one. The answer will be given in a later section.

For a stochastic system, a subset of the state space is generally hard to be (almost sure) invariance because the diffusion coefficient is required to be zero at the boundary of the subset\footnote{The detail is discussed in\cite{nishimura2016scl}, which aims to make the state of a stochastic system converge to the origin with probability one and confine the state in a specific subset with probability one. The aim is a little like the aim of a control barrier function. {Tamba et al. make a similar argument for CBFs in \cite{tamba2021}, but their sufficient condition is more stringent.}}. To avoid the tight condition for the coefficient, we should design a state-feedback law whose value is massive, namely diverge in general, at the boundary of the subset so that the effect of the law overcomes the disturbance term. Moreover, a functional ensuring the (almost sure) invariance of the subset probably diverges at the boundary of the set as with a global stochastic Lyapunov function \cite{khasminskii2012,kushner,mao2007} and an RCBF.

The above discussion also implies that if a ZCBF is defined for a stochastic system and ensures ``safety with probability one,'' the good robust property of the ZCBF probably gets no appearance. The reason is that the related state-feedback law generally diverges at the boundary of the safe set. Hence, the previous work in \cite{wang2021} proposes a ZCBF { with analysis of exit time of 
a state from a safe set}. In the next section, we consider another way to construct a ZCBF for a stochastic system; especially, we propose two types of ZCBFs; an almost sure ZCBF (AS-ZCBF) and a stochastic ZCBF, which have somewhat different conditions compared with ZCBFs in \cite{clark2021} and \cite{wang2021}. Then, in Section~\ref{sec:example}, we confirm the usefulness of our ZCBFs for control design by a few examples with numerical simulation.

\section{Main Claim}\label{sec:main}

\subsection{Definitions of a safe set and safety for a stochastic system}
Let us define a safe set $\chi \subset \R^n$ being open, and there exists a mapping $h:\R^n \to \R$ satisfying all the following conditions:
\begin{description}
\item[(Z1)] $h(x)$ is $C^2$ {for $x \in \chi$}. 
\item[(Z2)] $h(x)$ is proper in $\chi$; that is, for any $L \in [0,\infty)$, any superlevel set $\{x \in \chi | h(x) \ge L\}$ is compact.
\item[(Z3)] The closure of $\chi$ is the $0$-superlevel set of $h(x)$; that is,
\begin{align}
&\chi = \{x \in \R^n | h(x) > 0\},\\
&\partial \chi = \{ x \in \R^n | h(x)=0\}, 
\end{align}
are both satisfied.
\end{description}
If needed, (Z2) is sometimes replaced by the following:
\begin{description}
\item[(Z2)'] $h(x)$ is proper in $\R^n$.
\end{description}

We also notice that the reciprocal function $B(x) := (h(x))^{-1}$ is often used after.

{
Here, we set some sets and stopping times used in this subsection. For $\mu > 0$, let
\begin{align}
&\label{eq:safe_mu} \chi_\mu := \{x \in \R^n | h(x) \in (0, \mu ]\} \subset \chi, \\
&\chi_{h > \mu} := \chi \setminus \chi_\mu = \{x \in \R^n | h(x) > \mu \}, \\
&\R^n_{h \le \mu}:=  \{x \in \R^n | h(x) \le \mu\},
\end{align}
be defined. For a solution to the system \eqref{eq:sys-sto-gen} with $x_0 \in \chi_\mu$, the first exit time from $\chi_\mu$ is denoted by $\tau_{0\mu}$, and for the solution with $x_0 \in \chi$, the first exit time from $\chi$ is denoted by $\tau_0$.
}

Let $p \in [0,1]$. System \eqref{eq:sys-sto-gen} is said to be {\it {transiently} safe in} $({\chi_\mu},\chi,p)$ if 
\begin{align}
{ \pri{x_0}{ \sup_{t \ge 0} h(X_{t \wedge \tau_{0\mu}}) > 0 \Big| x_0 \in \chi_\mu } \ge p} 
\end{align}
is satisfied. {Moreover, if $\tau_{0\mu}=\tau_0$ holds, the system is said to be safe in $(\chi_\mu,\chi,p)$.}

\subsection{CBFs ensuring almost sure safety}

In this subsection, we describe sufficient conditions for $h(x)$ and $B(x)$ to ensure that the target system is FIiP in $\chi$. 

{First, we consider a reciprocal type of CBF for safety with probability one:}
\begin{definition}[AS-RCBF]
Let \eqref{eq:sys-sto-gen} be considered with $\chi$ and $h(x)$ satisfying (Z1), (Z2) and (Z3). Let also $x_0 \in \chi$ be assumed. If there exist a continuous mapping $\phi: \chi \to \R^m$ and a constant $\gamma>0$ such that, for all $x \in \chi$, 
\begin{align}\label{eq:con-rcbf}
\mathcal{L}_{f,g,\sigma}(\phi(x),u_o(x),B(x)) &\le \gamma B(x) 
\end{align}
is satisfied, then $B(x)$ is said to be {\it an almost sure reciprocal control barrier function (AS-RCBF)}. \eod
\end{definition}

{The existence of an AS-RCBF ensures the target system is safe with probability one because the following theorem is derived:}
\begin{theorem}\label{thm:rcbf}
If there exists an AS-RCBF $B(x)$ for the system \eqref{eq:sys-sto-gen}, then it is FIiP in $\chi$. \eot
\end{theorem}


The above condition \eqref{eq:con-rcbf} is more relaxed than the condition of a stochastic RCBF shown in \cite{clark2021} {because the value of $\mathcal{L}_{f,g,\sigma}(\phi(x),u_o(x),B(x))$ is allowed to be large near the boundary of the safe set ($B(x) \to \infty$ as $x$ tends to $\partial \chi$ from the inner), while $\mathcal{L}_{f,g,\sigma}(\phi(x),u_o(x),B(x)) \le 0$ is required at $x \in \partial \tilde{\chi}$ in \cite{clark2021}. Moreover, our condition} is similar to the condition of an extended RCBF for a deterministic system proposed in\cite{furusawa2021,nakamura2019}. The dual notion of the AS-RCBF is defined as follows.

\begin{definition}[AS-ZCBF]
Let \eqref{eq:sys-sto-gen} be considered with $\chi$ and $h(x)$ satisfying (Z1), (Z2) and (Z3). Let also $x_0 \in \chi$ be assumed. If there exist a continuous mapping $\phi: \chi \to \R^m$ and a constant $\gamma>0$ such that, for all $x \in \chi$, 
\begin{align}\label{eq:con-zcbf}
\mathcal{L}_{f,g,\sigma}(\phi(x),u_o(x),h(x)) \ge &-\gamma h(x) + L^I_{\sigma}(h(x)) \nonumber \\
&+ (h(x))^2 L^I_{\sigma}(B(x)) 
\end{align}
is satisfied, then $h(x)$ is said to be {\it an almost sure zeroing control barrier function (AS-ZCBF)}. \eod
\end{definition}

{The above definition of an AS-ZCBF is proposed to derive the following result:}
\begin{theorem}\label{thm:zcbf}
If there exists an AS-ZCBF $h(x)$ for system \eqref{eq:sys-sto-gen}, then it is FIiP in $\chi$. \eot
\end{theorem}


Next, we show a control design of $u=\phi(x)$ using an AS-RCBF and an AS-ZCBF.
\begin{corollary}\label{cor:ctrl-zcbf}
Consider the system \eqref{eq:sys-sto-gen}, the safe set $\chi$, $h(x)$ and $B(x)$ satisfying all the conditions of (Z1)--(Z3). Let
\begin{align}
&I(u_o(x),h(x)) := L^D_{f,g}(0,u_o(x),h(x)) \label{eq:I} \\
&J(h(x)) := -\gamma h(x) + (h(x))^2 L^I_{\sigma}(B(x)) \label{eq:J}
\end{align}
and
\begin{align}
&\phi_N(x) := \nonumber \\
&\left\{ \begin{array}{ll}
-\frac{I(u_o(x),h(x))-J(h(x))}{\lie{g}{h}(x) (\lie{g}{h}(x))^T} (\lie{g}{h}(x))^T, & I < J \cap \lie{g}{h} \neq 0\\
0 , & I \ge J \cup \lie{g}{h} = 0
\end{array}\right.
\end{align}
be designed. If 
\begin{align}\label{eq:ctrl-con}
\lie{f}{h}(x) > - \gamma h(x) + (h(x))^2 L^I_\sigma(B(x))
\end{align}
holds for $\lie{g}{h} = 0$, then the compensator $u=\phi_N(x)$ yields that the system \eqref{eq:sys-sto-gen} is FIiP in $\chi$. \eot
\end{corollary}

\begin{remark}
The control design in Corollary~\ref{cor:ctrl-zcbf} is a stochastic version of the control design in \cite{tezuka2022}. As in the literature, we can probably discuss optimality of a stochastic system \eqref{eq:sys-sto} with $u=\phi_N(x)$. The issue is out of the scope of this paper; it will be left as a topic for future work. \eor
\end{remark}

\begin{remark}
If the condition \eqref{eq:con-zcbf} becomes strict; i.e., ``$\ge$'' is replaced by {``$>$''}, the additional condition \eqref{eq:ctrl-con} obviously holds. \eor
\end{remark}

\subsection{A Stochastic ZCBF {and Safety-critical Control Design}}

In this subsection, we propose a new type of a ZCBF for a stochastic system to yield a quantitative evaluation of how safe the system is from the viewpoint of probability. { Then, we propose a design procedure for constructing a state-feedback law based on our ZCBF.} 

We propose the following notion of a stochastic ZCBF: 
\begin{definition}[Stochastic ZCBF]\label{def:szcbf}
Let \eqref{eq:sys-sto-gen} be considered with $\chi$ and $h(x)$ satisfying (Z1), (Z2)' and (Z3). If there exists a continuous mapping $\phi: \R^n \to \R^m$ such that, for all $x \in \R^n_{h \le \mu}$, 
\begin{align}\label{eq:con-prob}
\mathcal{L}_{f,g,\sigma}(\phi(x),u_o(x),h(x)) \ge b H_{\sigma}(h(x)) 
\end{align}
is satisfied with some $b>0$, then $h(x)$ is said to be {\it a stochastic ZCBF}. \eod
\end{definition}

{The quantitative evaluation of the safety probability is specifically given by the following result:}
\begin{theorem}\label{lem:safe-prob}
If there exists a stochastic ZCBF $h(x)$ for the system \eqref{eq:sys-sto-gen}, then it is {transiently} safe in $(\chi_\mu,\chi,1-e^{-b h(x_0)})$. {Moreover, if \eqref{eq:con-prob} is satisfied for any $\mu>0$, then the system is safe in $(\chi,\chi,1-e^{-bh(x_0)})$. }\eot
\end{theorem}

{
Next, we show a control design of $u=\phi_s(x)$ using a stochastic ZCBF.
\begin{corollary}\label{cor:ctrl-szcbf}
Consider the system \eqref{eq:sys-sto-gen}, the safe set $\chi$ and a candidate of a stochastic ZCBF $h(x)$ satisfying all the conditions of (Z1), (Z2)' and (Z3), and $\chi_\mu$ with $\mu>0$ as with \eqref{eq:safe_mu}. Let
\begin{align}
&I_s(u_o(x),h(x)) := {\mathcal{L}}_{f,g,\sigma}(0,u_o(x),h(x)) \label{eq:I-sto} \\
&J_s(h(x)) := b H_\sigma(h(x)) \label{eq:J-sto}
\end{align}
and
\begin{align}
&\phi_s(x) := \nonumber \\
&\left\{ \begin{array}{ll}
-\frac{I_s(u_o(x),h(x))-J_s(h(x))}{\lie{g}{h}(x) (\lie{g}{h}(x))^T} (\lie{g}{h}(x))^T, & I_s < J_s \cap \lie{g}{h} \neq 0\\
0 , & I_s \ge J_s \cup \lie{g}{h} = 0
\end{array}\right.\label{eq:ctrl-szcbf-ori}
\end{align}
be designed. If
\begin{align}\label{eq:ctrl-con-prob}
\lie{f}{h}(x) + L^I_\sigma(h(x)) > b H_\sigma(h(x))
\end{align}
holds for all $x \in \chi_\mu$ satisfying $\lie{g}{h} = 0$, then the compensator 
\begin{align}\label{eq:ctrl-szcbf-result}
u = \left\{ \begin{array}{ll} \phi_s, & x \in \R^n_{h\le\mu}, \\ \phi'_s, & x \in \chi_{h>\mu}, \end{array} \right.
\end{align}
where $\phi'_s:\chi_{h>\mu} \to \R^m$ is continuous and satisfies $\phi'_s(x)=\phi_s(x)$ for all $x \in \partial \chi_{h>\mu}$, yields that the system \eqref{eq:sys-sto-gen} is transiently safe in $(\chi_\mu, \chi, 1-e^{-b h(x_0)})$. Moreover, if the above discussion is satisfied for any $\mu>0$, the system is safe in $(\chi,\chi,1-e^{-b h(x_0)})$. \eot
\end{corollary}
}

\begin{remark}
A characteristic feature of our stochastic ZCBF in Definition~\ref{def:szcbf} appears in the condition \eqref{eq:con-prob} that includes the diffusion coefficient $\sigma(x)$ explicitly, which differs from the main previous study such as \cite{prajna2007,santoyo2019,wisniewski2021}. { The benefit appears in the control design; for example, if $x,w \in \R$, we obtain $J_s = (\sigma(x))^2 b/2 (\partial^2 h / \partial x^2)$ that is directly included in the control law \eqref{eq:ctrl-szcbf-ori}. For many control application problems, $\sigma(x)$ has modeling errors or varies depending on the experimental environment. To deal with these situations, in \eqref{eq:ctrl-szcbf-ori}, we can redesign $\sigma(x)$ according to the assumed error or variation. } \eor 
\end{remark}

In addition, because the condition \eqref{eq:exp-szcbf}{, which will appear later in the proof of Lemma~\ref{lem:safe-prob} in Appendix~\ref{subsec:lemma-safe-prop},} implies that $B_b(x)$ is non-negative supermartingale \cite{khasminskii2012,mao2007} outside of the safe set $\chi$, sample paths approach the safe set $\chi$ in probability. More concretely, we can employ the analysis of $\mu$-zone mean-square convergence shown in \cite{poznyak2018}. Letting $\mu_b := e^{-b\mu}$ and 
\begin{align}
&[\ex{B_b({X_t})} - \mu_b]_{+} := \left\{ \begin{array}{ll}
\ex{B_b({X_t})}-\mu_b, & \ex{B_b({X_t})} \ge \mu_b \\
0, & \ex{B_b({X_t})} < \mu_b
\end{array}\right., \\
&{V}(x) := ([\ex{B_b({X_t})} - \mu_b]_{+})^2, 
\end{align}
then the following holds.

\begin{corollary}\label{cor:mu-zone}
If there exists a stochastic ZCBF $h(x)$ for \eqref{eq:sys-sto-gen}, and moreover, for all $x \in \R^n_{h \le \mu}$, the condition \eqref{eq:con-prob} is replaced by
\begin{align}\label{eq:con-prob-strict}
\mathcal{L}_{f,g,\sigma}(\phi(x),u_o(x),h(x)) > b H_{\sigma}(h(x)), 
\end{align}
then, for solutions of the system with $u=\phi(x)$ and $x_0 \in \R^n_{h \le \mu}$,
\begin{align}
\lim_{t \to \infty} {V}({X_t}) = 0
\end{align}
is satisfied. \eot
\end{corollary}

\section{Examples}\label{sec:example}

\subsection{Revisit to the motivating example}\label{subsec:ex1}
In this subsection, we revisit the motivating example dealt with in Section~\ref{sec:motivation}. Let us consider the stochastic system \eqref{eq:sys-sto}, provided that a safe set is $\chi_1=(\alpha,\infty)$ according to Section~\ref{sec:main} {and the pre-input $u_o(x)$ is added; that is, $dX_t = (u_o(X_t)+u(t)) dt + c dW_t$}.

First, we remake the { CBFs} $B_s(x)=1/(x-\alpha)$ and $h_s=x-\alpha$ so that it is proper in $\chi_1$. Referring to \cite{nakamura2019}, set
\begin{align}
p_N(x) := \frac12 (x-N)^4 + \frac12 (x-N)^3|x-N|
\end{align}
for a sufficiently large $N>\alpha$. Then, the functions
\begin{align}
&B_1(x) = \frac{1}{x-\alpha} + p_N(x),\\
&h_1(x)=(B_1(x))^{-1}= \frac{x-\alpha}{1 + (x-\alpha)p_N(x)},
\end{align}
are proper in $\chi_1$. The shape of $h_1(x)$ is shown in Fig.~\ref{fig:h1shape}. 

Here, we consider $u=\phi_N(x)$ defined in Corollary~\ref{cor:ctrl-zcbf}. If $x \ge N$, $\phi_N$ is somewhat complicated because $p_N(x) = (x-N)^4$. However, if $x < N$, the calculation results in Section~\ref{sec:motivation} can be used because $p_N(x)=0$. That is, $\phi_N$ for any $x < N${, $\phi_N= -u_o(x) + \phi_{B_s}$ satisfying $I < J$, otherwise $\phi_N=0$.} Thus, we conclude $B_1$ and $h_1$ are an AS-RCBF and an AS-ZCBF, respectively.

Next, we consider the same problem setting as above, provided that the amplitude of the input is bounded; that is, for some $U_M>0$, an extra condition {$-U_M \le u_o(x) +u \le U_M$} is considered. {Let} a compensator $u=\phi_1(x)$ by designing
{
\begin{align}
\phi_1(x) := \left\{ \begin{array}{ll}
U_M, & h_s \le 0 \cup \phi_{B_s} > U_M \\
-U_M, & h_s > 0 \cap \phi_{B_s} < - U_M \\
\phi_{B_s}(x), & h_s > 0 \cap |\phi_{B_s}| \le U_M 
\end{array}\right. 
\end{align}
if $I<J$, otherwise, $\phi_1(x)=0$ if $I \ge J$} and $h_1$ be considered as a candidate of a stochastic ZCBF. {Then,} we obtain 
\begin{align}
{\mathcal{L}_{0,1,c}}(\phi_1(x),{0},h_1) {\ge b_1} H_{c}(h_1(x)) 
\end{align}
with {$b_1 = 2/\mu_1 + 2 \gamma \mu_1/c^2$ and $\mu_1 < c/\sqrt{\gamma}$.} 
The above results imply that $h_1$ is a stochastic ZCBF and the system is {transiently} safe in {$(\chi_{\mu_1},\chi,1-e^{-b_1 h_1(x_0)})$}. 

{
On the other hand, using Corollary~\ref{cor:ctrl-szcbf}, $\phi_s(x)= b c^2/2$ is derived. Because $\mathcal{L}_{0,1,c}(u_o(x),\phi_s(x),h_1(x)) \ge b c^2 /2$ is satisfied for all $x \in \R$, the system with $u=\phi_s(x)$ is safe in $(\chi_1,\chi_1,1-e^{-b h(x_0)})$. Moreover, if the input constraint $|u| \le U_M$ exists, we have to restrict $b \in (0,2 U_M / c^2]$, which affects the safety probability $1-e^{-bh(x_0)}$.
}

\subsection{Confinement in a bounded subset}\label{subsec:ex2}

{

In this subsection, we consider a stochastic nonlinear system
\begin{align}\label{sys:chained}
dX_t = \sum_{j=1}^2 g_j(X_t) (u_{oj}(X_t) + u_j(t)) dt + G dW_t,\ j=1,2,
\end{align}
where {$|u_{oj}(x)+u_j| \le U_j$ with $j=1,2$ for some $U_1,U_2>0$ and } 
\begin{align}
g_1(x) = \begin{bmatrix}
1 \\ 0 \\ x_2
\end{bmatrix},\ g_2(x)= \begin{bmatrix}
0 \\ 1 \\  {-x_1}
\end{bmatrix},
G = \begin{bmatrix}
c_1 \\ 0 \\ c_2
\end{bmatrix},\ c_1,c_2 \in \R.
\end{align}
If $G=0$, the system is said to be a {Brockett integrator,} which is a typical model appearing in various real nonholonomic systems such as a two-wheeled vehicle robot \cite{bloch2015}. 

Let a candidate of a stochastic ZCBF $h_2$ and a safe set $\chi_2$ by
\begin{align}
&{h_2(x) = M - 2 x_3^2 + \frac12 X(x) (1+x_3^2) - 2^{\frac{x_3^2}{2}} (X(x))^{1+\frac{x_3}{2}}} \\
&\chi_2 = \{x \in \R^3 | h_2(x) > 0 \},
\end{align}
where {$X(x) = x_1^2+x_2^2$ and} $M > 0$. { Note that $h_2(x)$ is proper in $\R^3$ and $\lie{g}{h_2}\neq 0$ for all $x \in \R^3 \setminus \{0\}$; $V_2(x)=M-h_2(x)$ is a stochastic control Lyapunov function for \eqref{sys:chained} proposed in \cite{nishimura2013ieice}.} Therefore, {using Corollary~\ref{cor:ctrl-szcbf}}, we design $u=\phi_2(x)$ by
\begin{align}
\phi_2(x) := \left\{ \begin{array}{ll}
\Phi_2(x), & h(x) \le \mu_2, \\
\Phi_2(x) \frac{h_2(x)-M'}{\mu-M'}, & h(x) \in (\mu_2, M'),\\
0, & h(x) \ge M',
\end{array}\right.
\end{align}
where { $\Phi_2(x)$ is designed by the same way to $\phi_s(x)$ in \eqref{eq:ctrl-szcbf-ori} with $b=b_2>0$, $0 < \mu_2 \le M' < M$} and $\mu_2$ is designed so that $|\Phi_{2j}(x)| \le U_{2j}$ is satisfied for all $x \in \{ x \in \R^3 | h(x) = \mu_2\}$ and $j=1,2$. 

Applying the safety-critical control $u=\phi_2(x)$ to the system \eqref{sys:chained}, we obtain \eqref{eq:con-prob} with $b=b_2$. Therefore, we conclude that the resulting system is {transiently} safe in {$(\chi_{\mu_2},\chi_2,1-e^{-b_2h(x_0)})$}.

{Moreover, we consider $u=\Phi(x)$ for all $x \in \chi$ without input constraints. Because $\lie{g}{h_2}(x)=0$ holds for only $x=0$, we calculate \eqref{eq:ctrl-con-prob} just for $x=0$ and obtain $c_1^2 - 4 c_2^2 > 0$; if the inequality is satisfied, the system \eqref{sys:chained} with $u=\Phi(x)$ is safe in $(\chi_2,\chi_2,1-e^{-b_2 h(x_0)})$.
}

\begin{figure}[!t]
\begin{minipage}[t]{0.48\hsize}
\centering
\includegraphics[keepaspectratio, scale=0.3]{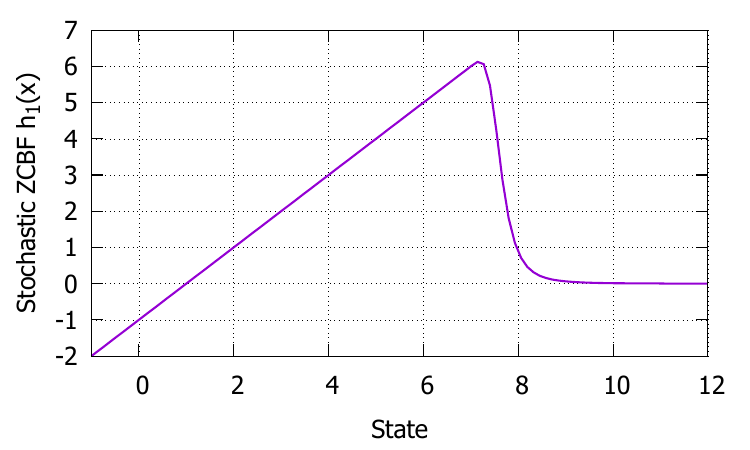}
\subcaption{The shape of $h_1(x)$ with $\alpha=1$ and $N=7$.}
\label{fig:h1shape}
\end{minipage}
\begin{minipage}[t]{0.48\hsize}
\centering
\includegraphics[keepaspectratio, scale=0.2]{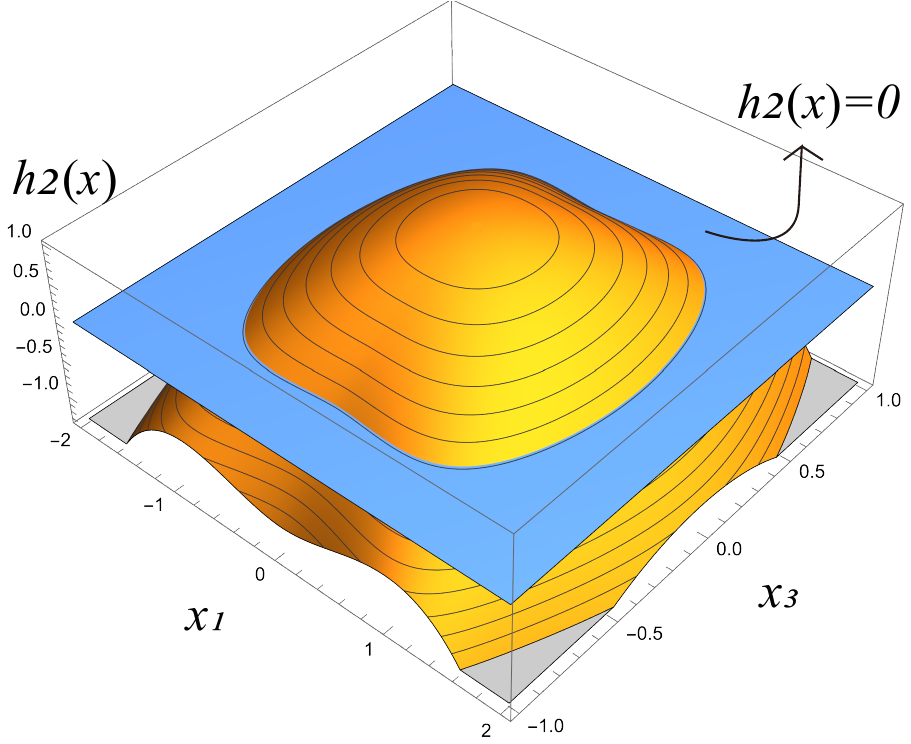}
\subcaption{The shape of $h_2(x)$ for {$x_2=0$ with $M=1$.}}
\label{fig:h2shape}
\end{minipage} \\
\caption{The shape of stochastic ZCBFs $h_1(x)$ and $h_2(x)$. The safe sets are all states for which $h_1(x)>0$ and $h_2(x)>0$, respectively. From the figures, we can see that $h_1(x)$ and $h_2(x)$ are proper, respectively.}
\end{figure}

\subsection{Numerical simulation}
In this subsection, we confirm the validity of the derived compensators for $\chi_1$ and $\chi_2$ by computer simulation.

\begin{example}\label{ex:ex1sim}
Consider the system \eqref{eq:sys-sto} with the safe set $\chi_1$ and the compensator $u=\phi_1$ discussed in Subsection~\ref{subsec:ex1}. Letting $\alpha=1$, {$\gamma = 0.5$}, $c=0.1${, $U_M=1$ and $\mu_1 = 0.13$}, we obtain {$x_{\mu_1} =1.13$ and $b_1 = 3.0$}. The value of $N$ does not affect computer simulation if we design it massively; for example, we set $N=10^{10}$. Then, { setting $x_0=1.06$,} the system \eqref{eq:sys-sto} is { transiently} safe in {$(\{ x \in (1, x_{\mu_1}) \}, \{ x > 1 \}, 0.96)$}. 
{The} compensator $\phi_1$ is illustrated as in Fig.~\ref{fig:u1}. {The} simulation results of time responses of the state $x$, the compensator $u$ and the pre-input $u_o$, and the ZCBF $h_1$ are described in Figs.~\ref{fig:ex1-x}, \ref{fig:ex1-u} and \ref{fig:ex1-zcbf}, respectively. In the simulation, we calculate ten times sample paths in the grey lines, the average of the paths in the red line, the results for the deterministic system (i.e., $\sigma'=0$ ) in the blue line, and the pre-input $u_o$ in the green line.
\end{example}

\begin{example}\label{ex:ex2sim}
{
Consider the system \eqref{sys:chained} with the safe set $\chi_2$ and the compensator $u=\phi_2$ discussed in Subsection~\ref{subsec:ex2}. Letting {$c_1=c_2=0.1$, $M=1$, $M'=0.95$, $b_2=4$, $\mu_2=0.83$ and $x_0=(0.5,0.5,0.2)^T$}, the system \eqref{sys:chained} is {transiently} safe in $( \chi_{\mu_2}, \chi_2,{0.92})$. {Assuming $u_o=(1,1)^T$}, the simulation results of the time responses of the compensators $u_1$ and $u_2$ are described in Figs.~\ref{fig:ex2-u1} and \ref{fig:ex2-u2}, respectively, and the time responses of the ZCBF $h_2$ are described in Fig.~\ref{fig:ex2-zcbf}. The colors of the lines have the same roles as in Example~\ref{ex:ex1sim}. 
}
\end{example}

In the simulation results, the safety is achieved better than the estimation of Theorem~\ref{lem:safe-prob}; note that the theorem ensures the minimum probability of leaving safe sets. The results may imply that, for actual control problems influenced by white noises, the designed compensators have good performances as safety-critical control.

\begin{figure}[!t]
\begin{tabular}{cc}
\begin{minipage}[t]{0.45\hsize}
\centering
\includegraphics[keepaspectratio, scale=0.3]{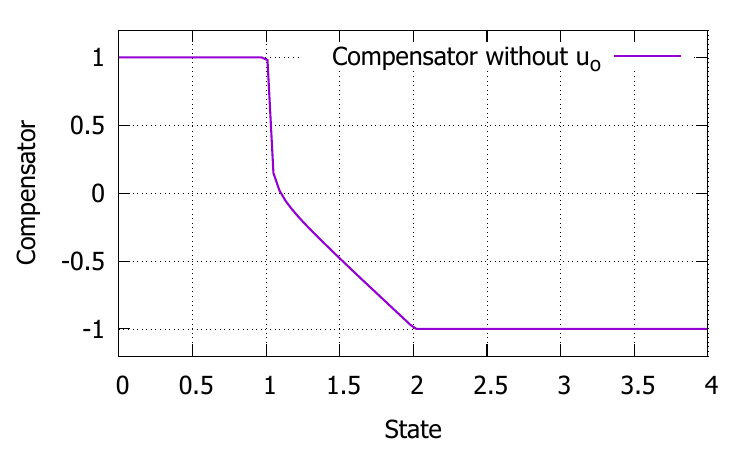}
\subcaption{The shape of the continuous state-feedback $\phi_1$.}
\label{fig:u1}
\end{minipage} &
\begin{minipage}[t]{0.45\hsize}
\centering
\includegraphics[keepaspectratio, scale=0.3]{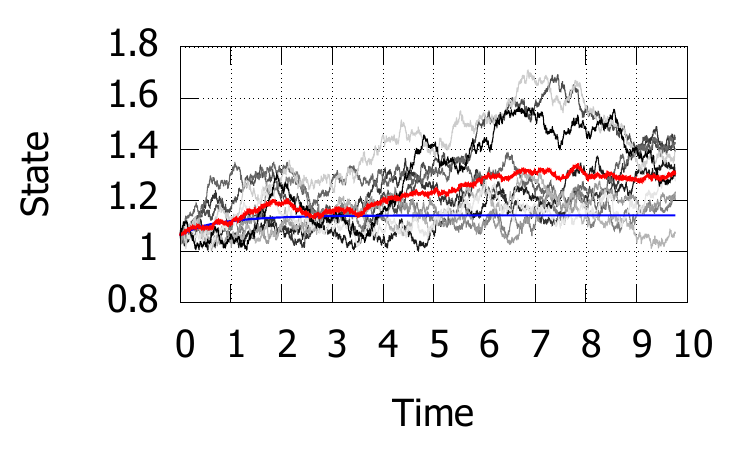}
\subcaption{Time responses of $x$.}
\label{fig:ex1-x}
\end{minipage} \\
\begin{minipage}[t]{0.45\hsize}
\centering
\includegraphics[keepaspectratio, scale=0.3]{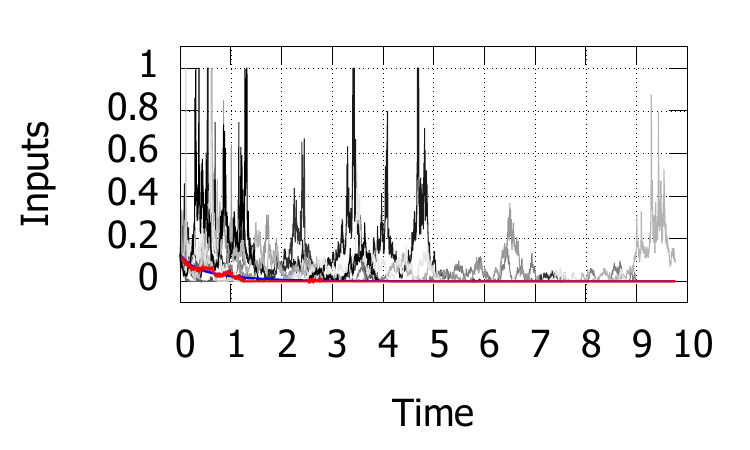}
\subcaption{Time responses of $\phi_1$.}
\label{fig:ex1-u}
\end{minipage} &
\begin{minipage}[t]{0.45\hsize}
\centering
\includegraphics[keepaspectratio, scale=0.3]{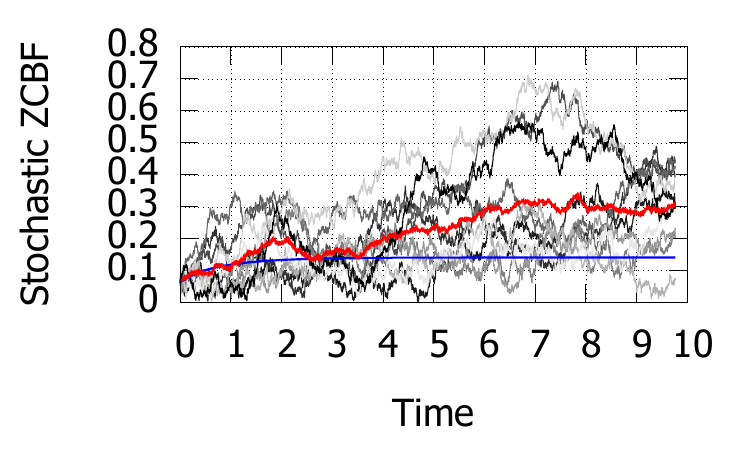}
\subcaption{Time responses of $h_1$.}
\label{fig:ex1-zcbf}
\end{minipage} 
\end{tabular}
\caption{Simulation results of Ex.~\ref{ex:ex1sim}. The grey colored lines denote 10 sample paths, the red colored lines denote the average of the paths, {and} the blue lines denote the results for the deterministic system (i.e., $\sigma'=0$ ). Fig.~\ref{fig:ex1-zcbf} shows that the safety condition $h_1(x)>0$ is satisfied in all 10 trials while the probability of the safety is {$0.96$}. }
\end{figure}

\begin{figure}[!t]
\centering
\begin{minipage}[t]{0.45\hsize}
\centering
\includegraphics[keepaspectratio, scale=0.3]{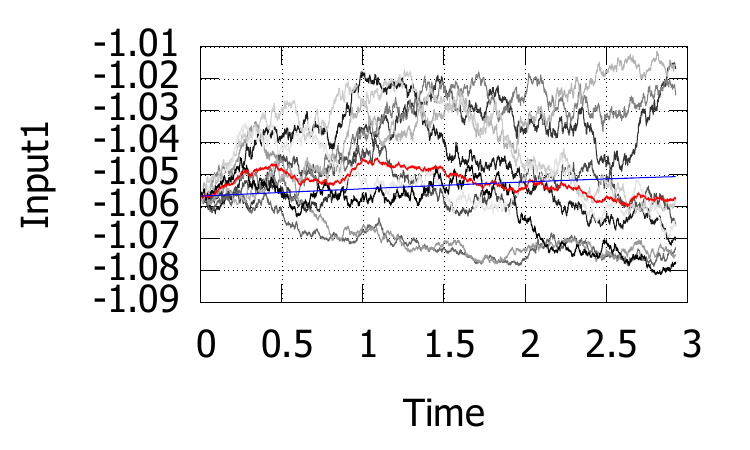}
\subcaption{Time responses of $u_1$.}
\label{fig:ex2-u1}
\end{minipage} \\
\begin{tabular}{cc}
\begin{minipage}[t]{0.45\hsize}
\centering
\includegraphics[keepaspectratio, scale=0.3]{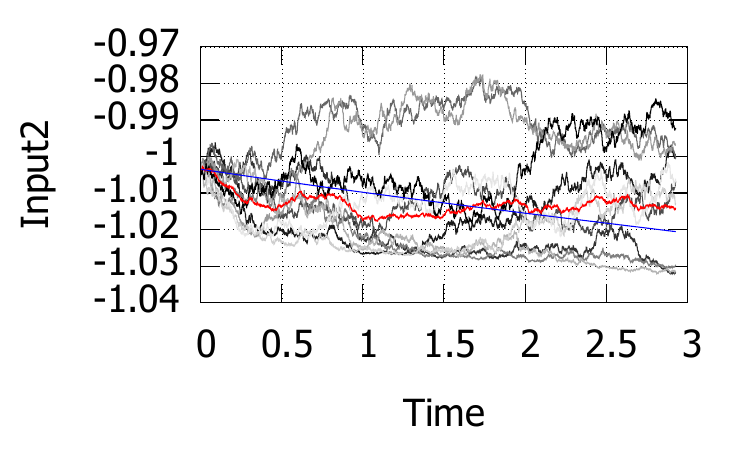}
\subcaption{Time responses of $u_2$.}
\label{fig:ex2-u2}
\end{minipage} &
\begin{minipage}[t]{0.45\hsize}
\centering
\includegraphics[keepaspectratio, scale=0.3]{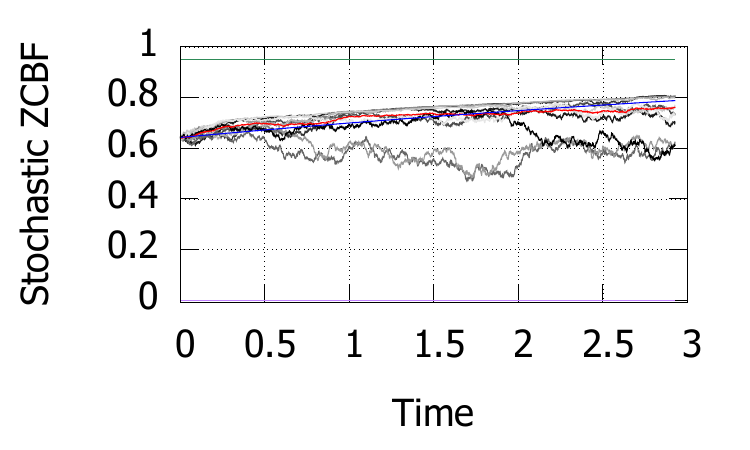}
\subcaption{Time responses of $h_2(x)$.}
\label{fig:ex2-zcbf}
\end{minipage} 
\end{tabular}
\caption{Simulation results of Ex.~2. The grey, red and blue lines are used in the same way as in Fig.~2. The purple and the green lines in Fig.~\ref{fig:ex2-zcbf} are the boundaries of {the set $\chi_{\mu_2}$}. Fig.~\ref{fig:ex2-zcbf} shows that the safety condition $h_2(x)>0$ is satisfied in { all} 10 trials while the probability of the safety is {$0.92$}.}
\end{figure}

\section{{Concluding Remarks}}\label{sec:conclusion}
In this paper, we proposed an almost sure reciprocal/zeroing control barrier function and a stochastic zeroing control barrier function for designing a safety-critical control law for a stochastic control system.  {We also show two examples to demonstrate the usefulness of the proposed method}. Because the target system is an input-affine stochastic system, the results can be extended to more general nonlinear control systems using, for example, the strategy of adding an integrator \cite{coron1991}. We also notice that the results are now effective for just an autonomous system with a state-feedback-type pre-input. The extension to a non-autonomous system with a time-varying pre-input is challenging for future work because the extension will enable us to apply our results for recent control application problems such as human-assist control \cite{furusawa2021,nakamura2019,tezuka2022}. Also, relaxing the constant that appears in the conditions for almost sure reciprocal/zeroing control barrier functions to the class $K_\infty$ function is essential, however, the relaxation requires rediscussing the existence of the solution. In addition, since the proposed method relaxes the conditions for the existence of solutions in continuous-time stochastic systems, sensitive discussions are needed to allow for discontinuous inputs, discontinuous dynamic variations, or more complex stochastic signals. Therefore, modifying the proposed method to support digital inputs, hybrid systems, and Poisson processes is a critical future task for its practical application.

{
\appendix

\section{Appendix: Proofs}

\subsection{Proof of Theorem~\ref{thm:rcbf}}
Applying the given condition \eqref{eq:con-rcbf} and (Z1)--(Z3) to Theorem~\ref{thm:fcip}, the system \eqref{eq:sys-sto-gen} is ensured to be FIiP in $\chi$. 

\subsection{Proof of Theorem~\ref{thm:zcbf}}
First, the condition \eqref{eq:con-zcbf} is transformed into
\begin{align}\label{eq:con-zcbf2}
L^D_{f,g}(\phi(x),u_o(x),h(x)) \ge -\gamma h(x) + (h(x))^2 L^I_{\sigma}(B(x)).
\end{align}
Using \eqref{eq:rel-bh} with $v=B$ and $M = \chi = \{ x \in \R^n | h(x)>0 \}$, the inequality is further transformed into
\begin{align}
-(B(x))^{-2} L^D_{f,g}(\phi(x),u_o(x),B(x)) \ge &-\gamma (B(x))^{-1} \nonumber \\
&+ (B(x))^{-2} L^I_{\sigma}(B) 
\end{align}
for $x \in M$. Thus, we obtain
\begin{align}
-(B(x))^{-2} \mathcal{L}_{f,g,\sigma}(\phi(x),u_o(x),B(x)) \ge -\gamma (B(x))^{-1},
\end{align}
which results in \eqref{eq:con-rcbf}. Therefore, the existence of an AS-ZCBF $B(x)$ ensures the system \eqref{eq:sys-sto-gen} is FIiP in $\chi$ via Theorem~\ref{thm:rcbf}.

\subsection{Proof of Corollary~\ref{cor:ctrl-zcbf}}
First, we consider the case of $\lie{g}{h} \neq 0$. If $I < J$, we obtain
\begin{align}
\mathcal{L}_{f,g,\sigma}&(\phi_N(x),u_o(x),h(x)) \nonumber \\
&= -\gamma h(x) +(h(x))^2 L^I_\sigma (B(x)) + L^I_\sigma (h(x)),
\end{align}
and if $I \ge J$, we obtain
\begin{align}
\mathcal{L}_{f,g,\sigma}&(\phi_N(x),u_o(x),h(x)) = I(u_o(x),h(x)) + L^I_\sigma (h(x)) \nonumber \\
&\ge 
-\gamma h(x) +(h(x))^2 L^I_\sigma (B(x)) + L^I_\sigma (h(x)).
\end{align}
Therefore, regardless of $I < J$ or $I \ge J$, the inequality \eqref{eq:con-zcbf} is satisfied. {Moreover, because $\lie{g}{h(x)}$, $I(u_o(x),h(x))$ and $J(h(x))$ are all continuous in $\lie{g}{h(x)} \neq 0$ and $\phi_N(x) \to 0$ as $I \to J$ uniformly when $\lie{g}{h(x)} \neq 0$, $\phi_N(x)$ is continuous in $\lie{g}{h(x)} \neq 0$. }

Then, we consider the other case, i.e., $\lie{g}{h} = 0$. 
The additional condition \eqref{eq:ctrl-con} implies that there exists a sufficiently small constant $\epsilon>0$ such that
\begin{align}
\lie{f}{h}(x) - \epsilon \ge - \gamma h(x) + (h(x))^2 L^I_\sigma(B(x))
\end{align}
is satisfied. Combining the inequality and the assumption of $u_o$ to be continuous, for a subset $G_o \subset \chi$, which is a neighborhood of $x_g \in \{ x \in \R^n | \lie{g}{h}(x)=0 \}$, 
\begin{align}
||\lie{g}{h}(x) u_o(x) || \le \epsilon
\end{align}
is satisfied. Thus, for $x \in G_o$, we obtain 
\begin{align}
\lie{f}{h}(x) + \lie{g}{h}(x) u_o(x) \ge - \gamma h(x) + (h(x))^2 L^I_\sigma(B(x)),
\end{align}
which implies that $I \ge J$; namely, $\phi_N=0$ in $G_o$. Therefore, $\phi_N$ is continuous around $\lie{g}{h}(x)=0$. 

Consequently, $\phi_N$ is always continuous in $\chi$ and satisfies all the assumptions and conditions of Theorem~\ref{thm:zcbf}. This completes the proof.

\subsection{Proof of {Theorem}~\ref{lem:safe-prob}}\label{subsec:lemma-safe-prop}
First, we prove that the existence of a stochastic ZCBF $h(x)$ ensures that the system \eqref{eq:sys-sto-gen} with $u=\phi(x)$ is FCiP. Let
\begin{align}\label{eq:func-prob}
h_b(x) := e^{bh(x)}.
\end{align}
Because 
\begin{align}
L^D_{f,g}(\phi(x),u_o(x),h_b(x)) = b h_b(x) L^D_{f,g}(\phi,u_o(x),h(x)),
\end{align}
is satisfied, \eqref{eq:con-prob} changes as follows:
\begin{align}\label{eq:con-prob3} 
L^D_{f,g}(\phi(x),u_o(x),h_b(x)) \ge b h_b(x) \left\{ b H_{\sigma}(h(x)) - L^I_{\sigma}(h(x))\right\}.
\end{align}
Moreover, letting
\begin{align}
B_b(x):=(h_b(x))^{-1}=e^{-bh(x)},
\end{align}
we obtain
\begin{align}\label{eq:htob}
L^I_{\sigma}(B_b(x)) = b B_b(x) \left\{ b H_{\sigma}(h(x)) - L^I_{\sigma}(h(x)) \right\},
\end{align}
which transforms \eqref{eq:con-prob3} into
\begin{align}\label{eq:con-prob4} 
L^D_{f,g}(\phi,u_o(x),h_b(x)) \ge (h_b(x))^2 L^I_{\sigma}(B_b(x)).
\end{align}
Therefore, remembering \eqref{eq:rel-bh} with $v=B_b$, we obtain
\begin{align}
-L^D_{f,g}(\phi,u_o(x),B_b(x)) \ge L^I_{\sigma}(B_b(x));
\end{align}
that is, 
\begin{align}\label{eq:exp-szcbf}
\mathcal{L}_{f,g,\sigma}(\phi,u_o(x),B_b(x)) \le 0,\ x \in \R^n_{h \le \mu}.
\end{align}

Here, we consider the rest space $\chi_{h > \mu}$, where the assumption (Z2)' implies that the space is bounded and $h$ is bounded from above in the space. In addition, $B_b$ is decreasing, $u_o$ is continuous, and $f$, $g$, and $\sigma$ are all locally Lipschitz. Therefore, $\mathcal{L}_{f,g,\sigma}(\phi,u_o(x),B_b(x))$ is bounded from above; that is, for sufficiently large values $c_1>0$ and $c_2>0$, we obtain
\begin{align}\label{eq:exp-szcbf-soto}
\mathcal{L}_{f,g,\sigma}(\phi,u_o(x),B_b(x)) \le c_1 B_b(x) + c_2,\ x \in \chi_{h > \mu}.
\end{align}
Considering \eqref{eq:exp-szcbf} and \eqref{eq:exp-szcbf-soto}, all the conditions of Theorem~\ref{thm:fcip} are satisfied with $Y=B_b$; that is, the system \eqref{eq:sys-sto-gen} with $u=\phi(x)$ is FCiP.

Next, going back to \eqref{eq:exp-szcbf} and applying Dynkin's formula \cite{khasminskii2012, kushner, mao2007}, provided that we restrict $x_0 \in \chi_\mu$, we obtain
\begin{align}
&\ex{B_b(X_{t\wedge \tau_{0\mu}})} - B_b(x_0) \nonumber \\
&= \mathbb{E} \Biggl[ \int_0^{t \wedge \tau_{0\mu}} \mathcal{L}_{f,g,\sigma}(\phi(X_\tau),u_o(X_\tau),B_b(X_\tau) d\tau \Biggr] \le 0.
\end{align}
{
Further considering $\pr{t \ge \tau_0} = \pr{t \ge \tau_0 \le \tau_\mu} + \pr{t \ge \tau_0 > \tau_\mu}$, $B_b(X_{t \wedge \tau_{0\mu}})=e^{-b\mu}$ for $t \ge \tau_0 > \tau_\mu$ and $\inf_{x \in \R^n \setminus \chi} B_b(x) = 1$, then we obtain
\begin{align}
\pri{x_0}{t \ge \tau_0 \le \tau_\mu} \le \exi{x_0}{B_b(X(t \wedge \tau_{0\mu}))},\ x_0 \in \chi_\mu.
\end{align}
Since $\pri{x_0}{\sup_{t \ge 0} B_b(X_{t \wedge \tau_\mu}) \ge 1 \cap \tau_0 > \tau_\mu} = 0$, we obtain
\begin{align}
\pri{x_0}{\sup_{t \ge 0}B_b(X_{t \wedge \tau_\mu}) \ge 1} &= \pri{x_0}{\sup_{t \ge 0} B_b(X_{t \wedge \tau_\mu}) \ge 1 \cap \tau_0 \le \tau_\mu} \nonumber \\
&=\pri{x_0}{t \ge \tau_0 \le \tau_\mu}.
\end{align}
Thus, we obtain
\begin{align}
\pri{x_0}{\sup_{t\ge 0}h(X_{t \wedge \tau_{0\mu}}) \le 0 } = \pri{x_0}{\sup_{t \ge 0}B_b(X_{t \wedge \tau_\mu}) \ge 1} \le B_b(x_0).
\end{align}
Therefore, the system is transiently safe in $(\chi_\mu,\chi,1-e^{-bh(x_0)})$. Moreover, if the discussion is satisfied for any $\mu>0$, we directly obtain 
\begin{align}
\pri{x_0}{\sup_{t \ge 0}B_b(X_t) \ge 1} = \pri{x_0}{t \ge \tau_0} &\le \exi{x_0}{B_b(X(t \wedge \tau_0)} \nonumber \\
&\le B_b(x_0), x_0 \in \chi.
\end{align}
}
This completes the proof.

\subsection{Proof of Corollary~\ref{cor:ctrl-szcbf}}
{
First, consider the case $\lie{g}{h} \neq 0$ in $\chi_\mu$. If $I_s < J_s$, we obtain
\begin{align}
\mathcal{L}_{f,g,\sigma}&(\phi_s(x),u_o(x),h(x)) = b H_\sigma(h(x))
\end{align}
and if $I_s \ge J_s$, we obtain
\begin{align}
\mathcal{L}_{f,g,\sigma}(\phi_s(x),u_o(x),h(x)) &= I_s(u_o(x),h(x)) \nonumber \\
&\ge J_s(h(x)) \nonumber = b H_\sigma(h(x)).
\end{align}
Therefore, regardless of $I_s < J_s$ or $I_s \ge J_s$, the inequality \eqref{eq:con-prob} is satisfied. Moreover, because $\lie{g}{h(x)}$, $I_s(u_o(x),h(x))$ and $J_s(h(x))$ are all continuous in $\lie{g}{h(x) \neq 0}$ and $\phi_s(x) \to 0$ as $I_s \to J_s$ uniformly when $\lie{g}{h(x)} \neq 0$, $\phi_s(x)$ is continuous in $\lie{g}{h(x)} \neq 0$. 
Then, we consider the other case, i.e., $\lie{g}{h} = 0$ in $\chi_\mu$. 
The additional condition \eqref{eq:ctrl-con-prob} implies that there exists a sufficiently small constant $\epsilon>0$ such that
\begin{align}
\lie{f}{h}(x) + L^I_\sigma(h(x)) - \epsilon \ge b H_\sigma(h(x))
\end{align}
is satisfied. Combining the inequality and the assumption of $u_o$ to be continuous, for a subset $G_{o\mu} \subset \chi_\mu$, which is a neighborhood of $x_g \in \{ x \in \chi_\mu | \lie{g}{h}(x)=0 \}$ \begin{align}
||\lie{g}{h}(x) u_o(x) || \le \epsilon
\end{align}
is satisfied. Thus, for $x \in G_{o\mu}$, we obtain 
\begin{align}
\lie{f}{h}(x) + L^I_\sigma(h(x)) + \lie{g}{h}(x) u_o(x) \ge b H_\sigma(h(x)),
\end{align}
which implies that $I_s \ge J_s$; namely, $\phi_s=0$ in $G_o$. Therefore, $\phi_s$ is continuous around $\lie{g}{h}(x)=0$ in $\chi_\mu$. 

Consequently, $\phi_s$ is always continuous in $\chi$ and satisfies all the assumptions and conditions of Theorem~\ref{lem:safe-prob}. Moreover, because $u=\phi'_s(x)$ is continuous in $\chi_{h>\mu}$ and $\phi'_s(x)=\phi_s(x)$ for all $x \in \partial \chi_{h>\mu}$, $u$ is continuous for all $\chi$. This completes the proof.
}

\subsection{Proof of Corollary~\ref{cor:mu-zone}}
Because the condition \eqref{eq:con-prob-strict} holds, for any $x \in \R^n_{h \le \mu}$, $B_b(x)=e^{-bh(x)}$ satisfies $\mu_b \le B_b(x)$. Using the inequality, we obtain
\begin{align}\label{eq:exp-szcbf-strict}
\mathcal{L}_{f,g,\sigma}(\phi,u_o(x),B_b(x)) < 0,\ x \in \R^n_{h \le \mu}
\end{align}
via the same way to derive \eqref{eq:exp-szcbf}. Therefore, using Dynkin's formula, we obtain
\begin{align}
\frac{dV}{dt} &= 2 [\ex{B_b(X_t)} - \mu_b]_{+} \frac{d\ex{B_b(X_t)}}{dt} \nonumber \\
&= 2 [\ex{B_b(X_t)} - \mu_b]_{+} \ex{\mathcal{L}_{f,g,\sigma}(\phi(X_t),u_o(X_t),B_b(X_t))} \nonumber \\
&<0.
\end{align}
This completes the proof.

}


\end{document}